\documentclass[12pt,notitlepage]{amsart}

\usepackage{graphicx}
\usepackage{amsmath, amsfonts, amssymb, amsthm}
\usepackage[myheadings]{fullpage}
\usepackage[utf8]{inputenc}
\usepackage[margin=1in]{geometry}
\usepackage{enumitem}
\usepackage{xcolor}
\usepackage{caption}
\usepackage{tabularx}
\usepackage{fancyhdr}
\usepackage{hyperref}
\usepackage{tikz}
\usetikzlibrary{angles,quotes}
\usepackage{mdframed}
\usepackage{lipsum}
\usepackage{chngcntr}
\usepackage{imakeidx}
\usepackage{tkz-euclide}
\usepackage{blindtext}
\usepackage[inline]{asymptote}
\usepackage{mathdots}
\usepackage{tikz-cd,mathtools}
\usepackage{ mathrsfs }
\usepackage{float}
\restylefloat{table}

\usepackage{pgfplots, pgfplotstable}
\usepackage{multicol}
\usepackage{endnotes}
\usepackage{idxlayout}
\usepackage{xcolor, forest}
\usepackage{tocvsec2}
\usepackage[toc]{glossaries}
\usepackage{venndiagram}
\usepgfplotslibrary{statistics}
\usepackage[colorinlistoftodos]{todonotes} 
\usepackage[normalem]{ulem}
\usetikzlibrary{shapes.geometric}
\usepackage{mathtools}

\numberwithin{equation}{section}

\newcommand{\Mod}[1]{\ \left(\mathrm{mod}\ #1\right)}
\newcommand{\igm}{\text{im}}

\theoremstyle{thmstyleone}%
\newtheorem{theorem}{Theorem}[section]
\newtheorem{lemma}[theorem]{Lemma}
\newtheorem{proposition}[theorem]{Proposition}%

\theoremstyle{thmstyletwo}%
\newtheorem{remark}[theorem]{Remark}%
\newtheorem{conjecture}[theorem]{Conjecture}%

\theoremstyle{thmstylethree}%

\numberwithin{theorem}{section}
\numberwithin{equation}{section}

\newcommand{\be}{\begin{equation}}
\newcommand{\ee}{\end{equation}}
\newcommand{\bea}{\begin{eqnarray}}
\newcommand{\eea}{\end{eqnarray}}

\newcommand{\qda}{q_d^{(a)}(n)}
\newcommand{\q}[3]{q_{#1}^{(#2)}\left(#3\right)}

\newcommand{\Qda}{Q_d^{(a)}(n)}

\newcommand{\Qdadash}{Q_d^{(a,-)}(n)}
\newcommand{\Qdash}[3]{Q_{#1}^{(#2,-)}\left(#3\right)}

\newcommand{\Qdadashdash}{Q_d^{(a,-,-)}(n)}
\newcommand{\Qdashdash}[3]{Q_{#1}^{(#2,-,-)}\left(#3\right)}

\newcommand{\Daa}{\Delta_d^{(a)}(n)}

\newcommand{\Ddasha}[3]{\Delta_{#1}^{(#2,-)}\left(#3\right)}

\newcommand{\Daadash}{\Delta_d^{(a,-)}(n)}

\newcommand{\Ddashdasha}[3]{\Delta_{#1}^{(#2,-,-)}\left(#3\right)}

\begin{document}


\title{On Generalizations of a Conjecture of Kang and Park}

\author{Ryota Inagaki}
\address{University of California, Berkeley}
\email{ryotainagaki@berkeley.edu}

\author{Ryan Tamura}
\address{University of California, Berkeley}
\email{rtamura1@berkeley.edu}

\date{\today}
\begin{abstract} 
Let $\Delta_d^{(a,-)}(n) = q_d^{(a)}(n) - Q_d^{(a,-)}(n)$ where $q_d^{(a)}(n)$ counts the number of partitions of $n$ into parts with difference at least $d$ and size at least $a$, and $Q_d^{(a,-)}(n)$ counts the number of partitions into parts $\equiv \pm a \pmod{d + 3}$ excluding the $d+3-a$ part. Motivated by generalizing a conjecture of Kang and Park, Duncan, Khunger, Swisher, and the second author conjectured that $\Delta_d^{(3,-)}(n)\geq 0$ for all $d\geq 1$ and $n\geq 1$ and were able to prove this when $d \geq 31$ is divisible by $3$. They were also able to conjecture an analog for higher values of $a$ that the modified difference function $\Delta_{d}^{(a,-,-)}(n) = q_{d}^{(a)}(n) - Q_{d}^{(a,-,-)}(n) \geq 0$ where $Q_{d}^{(a,-,-)}(n)$ counts the number of partitions into parts $\equiv \pm a \pmod{d + 3}$ excluding the $a$ and $d+3-a$ parts and proved it for infinitely many classes of $n$ and $d$. 

We prove that $\Delta_{d}^{(3,-)}(n) \geq 0$ for all but finitely many $d$. We also provide a proof of the generalized conjecture for all but finitely many $d$ for fixed $a$ and strengthen the results of Duncan et.al. We provide a conditional linear lower bound on d for the generalized conjecture by using a variant of Alder’s conjecture. Additionally, we obtain asymptotic evidence that this modification holds for sufficiently large n.
\end{abstract} 

\maketitle

\section{Introduction}\label{sec:introduction section}
A partition of a positive integer $n$ is a non-increasing sequence of positive integers, called parts, that sum to $n$.  Let $p(n \mid \mbox{condition})$ be the number of partitions of $n$ satisfying a certain condition. Euler famously proved that the number of partitions of a positive integer $n$ into odd parts equals the number of partitions of $n$ into distinct parts. Two other famous partition identities are those of Rogers and Ramanujan. The first Rogers-Ramanujan identity states that the number of partitions of $n$ with parts having difference at least $2$ is equal to the number of partitions of $n$ with parts congruent to $\pm 1 \pmod{5}$ and the second Rogers-Ramanujan identity states the number of partitions of $n$ with parts at least $2$ and difference at least $2$ is equal to the number of partitions of $n$ with parts congruent to $\pm 2 \pmod{5}$. These identities are encapsulated by the $q-$series
\begin{align*}
\sum_{n = 0}^{\infty} \frac{q^{n^{2}}}{(q\mbox{;}q)_{n}} &= \frac{1}{(q\mbox{;}q^{5})_{\infty}(q^{4}\mbox{;}q^{5})_{\infty}},\\
\sum_{n = 0}^{\infty} \frac{q^{n(n+1)}}{(q\mbox{;}q)_{n}} &= \frac{1}{(q^{2}\mbox{;}q^{5})_{\infty}(q^{3}\mbox{;}q^{5})_{\infty}},  
\end{align*}
where $(a\mbox{;}q)_{0} = 1$, and $(a\mbox{;}q)_{n} = \prod_{k=0}^{n-1}(1-aq^{k}) $, where $n = \infty$ is allowed. 

Motivated in generalizing the Rogers-Ramanujuan identities, Schur discovered the number of partitions with parts having difference at least $3$ which have no consecutive multiples of $3$ as parts is equal to the number of partitions with parts congruent to $\pm 1\pmod{6}$. 

Remarkably Alder \cite{alder_nonexistence_1948} and Lehmer \cite{lehmer_two_1946} showed there are no other partition identities similar to those of Euler, Rogers-Ramanujan, and Schur. In $1956$, Alder \cite{alder_research_1956} conjectured a generalization of a related family of partition identities. Alder's conjecture states that the number of partitions with parts that differ by at least $d$ is greater than or equal to the number of partitions with parts congruent to $\pm 1 \pmod{d+3}$. Note that this conjecture generalizes the Euler, Rogers-Ramanujan, and Schur identities. Alder's conjecture was proven by Andrews \cite{andrews_partition_1971} for $n \geq 1$ and $d = 2^{r} -1, r \geq 4$ in $1971$. In $2004$ and $2008$, Yee \cite{yee_partitions_2004,yee_alders_2008} proved the conjecture for $n \geq 1$, $d \geq 32$ and $d = 7$. In 2011, the remaining cases of Alder's conjecture were proven by Alfes, Jameson, and Lemke Oliver \cite{alfes_proof_2011} by using the asymptotic methods of Meinardus  \cite{meinardus_asymptotische_1954},  \cite{meinardus_uber_1954}. 

In $2020$, Kang and Park \cite{kang_analogue_2020} investigated how to construct an analog of Alder's conjecture that incorporates the second Rogers-Ramanujan identity. Kang and Park compared the partition functions
\begin{align*}
     &\qda := p(n \mid \text{ parts} \geq a \text{ and parts differ by at least } d), \\
    &\Qda := p(n \mid \text{ parts} \equiv \pm a \Mod{d+3}),
\end{align*} by utilizing the difference function
\begin{align*}
    \Delta_d^{(a)}(n) &:= q_{d}^{(a)}(n) -Q_{d}^{(a)}(n).
\end{align*}

In their attempts to create an analog of Alder's conjecture for the second Rogers-Ramanujan identity, Kang and Park found that 
\[
\Delta_{d}^{(2)}(n) < 0 \text{ for some choices of } d, n \geq 1.
\]
However, by employing a minor modification of $ Q_{d}^{(a)}(n)$ by defining for $d,a, n \geq 1$,
\begin{align*}
    \Qdash{d}{a}{n} &:= p(n\mid \text{ parts} \equiv \pm a \Mod{d+3}, \text{ excluding the part } d+3-a), \\
    \Ddasha{d}{a}{n}  &:= \q{d}{a}{n} - \Qdash{d}{a}{n},
\end{align*}
they presented the following conjecture. 
\begin{conjecture}[Kang, Park \cite{kang_analogue_2020}, 2020]\label{conj:KPconj}
    For all \(d\), \(n \geq 1\),
    \[
        \Ddasha{d}{2}{n} \geq 0.
    \]
\end{conjecture} 
 
Kang and Park's conjecture was proven for all but finitely many $d$ by Duncan, et al. by employing a modification of Alder's conjecture and the results of Andrews \cite{andrews_partition_1971} and Yee \cite{yee_alders_2008}.
\begin{theorem}[Duncan,  et al. \cite{ourpaper}, 2021] \label{thm:duncan, et.all}
For all $d \geq 62$ and $n \geq 1$,
\[
\Delta_{d}^{(2,-)}(n) \geq 0. 
\]
\end{theorem}
Motivated by Kang and Park's conjecture, Duncan et al. attempted to find a generalization for higher values of $a$. Remarkably, they observed that the removal of the $d +3-a$ as a part appeared to be sufficient when $a = 3$. 
\begin{conjecture}[Duncan,  et al. \cite{ourpaper}, 2021]\label{conj:a=3} 
For all $d,n \geq 1$, 
\[
\Delta_{d}^{(3,-)}(n) \geq 0.
\]
\end{conjecture}
By employing the methods of Duncan, et al. \cite{ourpaper}, we prove Conjecture \ref{conj:a=3} for all but finitely many $d$. 
\begin{theorem}\label{thm:a=3} 
For $n \geq 1$ and $d =1,2, 91,92,93$ or $d \geq 187$,
\[
\Delta_{d}^{(3,-)}(n) \geq 0.
\]
\end{theorem}

 When $a \geq 4$, it is not the case that $\Delta_{d}^{(a,-)}(n)$ is non-negative for all $d ,n \geq 1$. By considering the functions 
\begin{align*}
        \Qdadashdash &:= p(n \mid \text{ parts} \equiv \pm a \Mod{d + 3}, \text{excluding the parts } a \text{ and } d+3-a),
        \\
           \Ddashdasha{d}{a}{n} &:= q_{d}^{(a)}(n) - \Qdashdash{d}{a}{n},
\end{align*}
Duncan, et al. conjectured the following analog of Kang and Park's conjecture. 
\begin{conjecture}[Duncan, et al. \cite{ourpaper}, 2021] \label{conj:genkp}
For $d, a,n\geq $1 with $1 \leq a \leq d + 2$,
\[
\Delta_{d}^{(a,-,-)}(n) \geq 0. 
\]
\end{conjecture}
They were able to prove infinitely many cases of Conjecture \ref{conj:genkp} by employing the methods of Yee \cite{yee_alders_2008} and Andrews \cite{andrews_partition_1971}.

Throughout this paper, we let $h_{d}^{(a)} $ and $h_{n}^{(a)} $ denote the least non-negative residues of $-d$ and $-n$ modulo $a$ respectively. Note that $\frac{d + h_{d}^{(a)} }{a} = \lceil \frac{d}{a} \rceil$ and $\frac{n+ h_{n}^{(a)} }{a} = \lceil \frac{n}{a} \rceil$.

We prove a strengthening of \cite[Theorem 1.6]{ourpaper}. 
\begin{theorem}\label{thm:strengthen unc} 
 For $a \geq 4$ and $\lceil \frac{d}{a} \rceil = 31$ or  $\lceil \frac{d}{a} \rceil \geq 63$ such that $h_{d}^{(a)} \leq  3$, $d \not\equiv -3\pmod{a}$, and $n \geq 1$, 
\[
\Delta_{d}^{(a,-)}(n) \geq 0.
\]
Moreover, when $d \equiv -3\pmod{a}$, $\Delta_{d}^{(a,-)}(n) \geq 0$ for all $n \neq d+3+a$.

In particular, for $a =4$,\ $121 \leq d \leq 124$ or $d \geq 249$, and $n \geq 1$, $n \neq d +7$ when $d \equiv 1 \pmod{4}$,
\[
\Delta_{d}^{(4, -)}(n) = \q{d}{4}{n} -  \Qdash{d}{4}{n} \geq 0 .
\]
\end{theorem}
\begin{remark}\label{rem:strengthen-inequality}
This is a strengthening of \cite[Theorem 1.6]{ourpaper} since 
\[
\Delta_{d}^{(a,-,-)}(n) \geq \Delta_{d}^{(a,-)}(n)\geq \Delta_{d}^{(a)}(n) \text{ for all } d,a, n \geq 1. 
\]
\end{remark}
By developing a new method in comparing these partition functions, we also prove Conjecture \ref{conj:genkp} for all but finitely many $d$ for fixed $a$. 
\begin{theorem}\label{thm:genkp unc} 
For $d,n \geq 1$ and $a \geq 5$ such that $\lceil \frac{d}{a} \rceil \geq 2^{a+3} - 1$,
\[
\Delta_{d}^{(a,-,-)}(n) \geq 0. 
\]
\end{theorem}
\begin{remark}\label{rem:a not 5}
The cases of Conjecture \ref{conj:genkp} for $1 \leq a \leq 3$ have already been proven for all but finitely many $d$ by Alder's conjecture, Theorems \ref{thm:duncan, et.all} and   \ref{thm:a=3}. Note when $a = 4$ and $n = d + 7$ that Conjecture \ref{conj:genkp} is trivially verified. Hence, with Theorem \ref{thm:strengthen unc}, the case of Conjecture \ref{conj:genkp} for $a=4$ have been proven for all but finitely many $d$. Hence, we'll focus on proving Theorem \ref{thm:genkp unc}. 
\end{remark}
It is a natural question of determining for fixed $a \geq 1$ which $d,n \geq 1$ are sufficient to allow 
\[
\Delta_{d}^{(a,-)}(n) \geq 0. 
\]
Towards answering this question, we present the following Alder-type inequality. 
\begin{conjecture}\label{conj:aldermconj} 
For $d,n\geq 12$ such that $n \geq d +2$,
\[
q_{d}^{(1)}(n) - Q_{d-4}^{(1,-)}(n) \geq 0. 
\]
\end{conjecture}
\begin{remark}
We have learned that Armstrong, et.al \cite{armstrong_alder_type_2022} have proven Conjecture \ref{conj:aldermconj} for $n \geq 1$ and $d \geq 105$. They were also able to obtain a generalization of Conjecture \ref{conj:aldermconj}, showing for $N \geq 2$, $d\geq \max\{63, 46N - 79\} $, and $n \geq d + 2$, that
\[
q_{d}^{(1)}(n) \geq \Qdash{d -N}{1}{n}. 
\]
\end{remark}
Assuming Conjecture \ref{conj:aldermconj}, we are able to prove a strengthening of Conjecture \ref{conj:genkp} which allows $a$ as a part. This conditional result  remarkably reduces  the exponential lower bound of $d$ in Theorem \ref{thm:genkp unc} to a linear one. 
\begin{theorem}\label{thm:genkp con} 
Suppose that Conjecture \ref{conj:aldermconj} holds for the prescribed bounds. Then for $a \geq 1$,  $\lceil \frac{d}{a} \rceil \geq 12$, and $n \geq 5\lceil\frac{d}{a}\rceil + 1$,  
\[
\Delta_{d}^{(a,-)}(n) \geq 0.
\]
Moreover, suppose the bounds on $a$ and $d$ are as above, and $1\leq \lceil \frac{n}{a} \rceil \leq  5\lceil \frac{d}{a} \rceil  $, $d \not\equiv -3\pmod{a}$, unconditionally we have
\[
\Delta_{d}^{(a,-)}(n) \geq 0.
\]
If $d \equiv -3\pmod{a}$, then unconditionally $\Delta_{d}^{(a, -)}(n) \geq 0$ for all $1 \leq \lceil \frac{n}{a}\rceil \leq 5\lceil \frac{d}{a}\rceil$ and $n \neq d + 3+a$. 
\end{theorem}
\begin{remark}
The choice of the upper bound $1 \leq \lceil \frac{n}{a} \rceil \leq 5\lceil \frac{d}{a} \rceil $ for the unconditional component of Theorem \ref{thm:genkp con} comes from applying the function $\mathcal{G}_{d}^{(1)}(n)$ as defined by Yee \cite{yee_alders_2008}. 
\end{remark}

We now outline the rest of this paper. In Section \ref{sect:prelims}, we present a modification of \cite[Theorem 3]{andrews_partition_1971} and establish some other important lemmas. In Section \ref{sect:modalders}, we prove a modification of Alder's conjecture, which forms the critical component of the proofs of Theorems \ref{thm:a=3} and \ref{thm:strengthen unc}. In Section \ref{sect:a=3 and uncstrengthgenkp}, we prove Theorems \ref{thm:a=3} and \ref{thm:strengthen unc} by reducing to our result in Section \ref{sect:modalders}. In Section \ref{sect: generalized KP exp}, we prove Theorem \ref{thm:genkp unc} by using the partition counting functions from Yee \cite{yee_alders_2008}. In Section \ref{sect: generalized KP con}, we conditionally prove Theorem \ref{thm:genkp con} by reducing to the case of $a = 1$. We also prove Conjecture \ref{conj:aldermconj} for small $n$ to derive our unconditional result in Theorem \ref{thm:genkp con}. Finally, in Section \ref{sect:asymptotics}, we obtain asymptotic evidence that Proposition \ref{prop:mod alders} and Conjecture \ref{conj:aldermconj} holds for large $n$ and $d \geq 10$ and $d \geq 12$ respectively. We conclude with potential avenues for resolving more small $d$ cases of Conjectures \ref{conj:a=3}, \ref{conj:genkp} and extending Theorem \ref{thm:strengthen unc}.   

\section{Preliminaries}\label{sect:prelims}
In this section, we first establish generating functions for our partition counting functions. We then introduce several lemmas which our results are based on.

We find by employing combinatorial methods that the generating functions for \(\qda\) and \(\Qda\) for $1 \leq a \leq d + 2$ are
    \begin{align*}
        \sum_{n = 0}^{\infty} \qda q^n &= \sum_{k = 0}^{\infty} \frac{q^{d{k \choose 2} + ka}}{(q\mbox{;}q)_{k}}, \\
  \sum_{n = 0}^{\infty} \Qda q^n &= \frac{1}{(q^{d + 3 - a}\mbox{;}q^{d + 3})_{\infty}(q^{a}\mbox{;}q^{d + 3})_{\infty}}. 
\end{align*}

 As described in \cite{ourpaper}, we find by removing the $d+3-a$ term that the generating function for $\Qdash{d}{a}{n}$ is \begin{equation*}
        \displaystyle\sum_{n = 0}^{\infty} \Qdadash q^n =
        \begin{cases}
            \frac{1}{(q^{2d + 6 - a}\mbox{;}q^{d + 3})_{\infty}} &\text{ for } a = \frac{d + 3}{2},\\
            \frac{1}{(q^{2d + 6 - a}\mbox{;}q^{d + 3})_{\infty}(q^{a}\mbox{;}q^{d + 3})_{\infty}} &\text{ otherwise.} \label{eq:Qdbdashgenfunc}
        \end{cases}
    \end{equation*} 

Similarly, we find that the generating function of $\Qdashdash{d}{a}{n}$ to be
 \begin{equation*}
        \displaystyle\sum_{n = 0}^{\infty} \Qdadashdash q^n =
        \begin{cases}
        \frac{1}{(q^{2d + 6 - a}\mbox{;}q^{d + 3})_{\infty}} &\text{ for } a = \frac{d + 3}{2},
        \\
        \frac{1}{(q^{2d + 6 - a}\mbox{;}q^{d + 3})_{\infty}(q^{d + 3 + a}\mbox{;}q^{d + 3})_{\infty}} &\text{ otherwise.} \label{eq:Qdbdashdashgenfunc}
        \end{cases}
    \end{equation*}

We employ the following notation: for a set of positive integers $R$, we define \[
\rho(R\mbox{;}n):= p(n \mid \text{parts from the set } R).
\]

 Suppose that $\lambda$ is a partition. We let $\lambda \vdash n$ to denote that $\lambda$ is a partition of $n$. For fixed positive integers $d$ and $r$, define as in Andrews  \cite{andrews_partition_1971} the partition function $\rho(T_{r,d}\mbox{;} n)$  with set of parts 
\[
T_{r,d} = \{x \in \mathbb{N} \mid x \equiv 1, d +2, \cdots, d + 2^{r-1}\pmod{2d}\}. \]
The generating function of $\rho(T_{r,d}\mbox{;} n)$ is 
\begin{equation*}
 \sum_{n=0}^{\infty}\rho(T_{r,d}\mbox{;} n)q^{n} = \frac{1}{(q^{1}\mbox{;} q^{2d})_{\infty}(q^{d+2}\mbox{;} q^{2d})_{\infty}\cdots (q^{d+2^{r-1}}\mbox{;} q^{2d})_{\infty} }.
\end{equation*}

Suppose that $d \geq 1$. Throughout this paper, we define $r_d$ to be the largest positive integer $r$ such that $d \geq 2^r - 1$. 

We also use the partition function $\mathcal{G}_{d}^{(1)}(n)$ considered by Yee \cite{yee_alders_2008}, which is defined by 
\begin{equation}\label{eq:G function}
    \sum_{n=0}^{\infty}\mathcal{G}_{d}^{(1)}(n)q^n := \frac{(-q^{d+2^{r_d-1}}\mbox{;} q^{2d})_{\infty}} {(q^{1}\mbox{;} q^{2d})_{\infty}(q^{d+2}\mbox{;} q^{2d})_{\infty}\cdots (q^{d+2^{r_d-2}}\mbox{;} q^{2d})_{\infty}}.
\end{equation}

From (\ref{eq:G function}), we find that $\mathcal{G}_{d}^{(1)}(n)$ counts partitions with distinct parts from the set $\{x \in \mathbb{N} \mid x \equiv d + 2^{r_d-1}\pmod{2d}\}$ and unrestricted parts from the set
\[
T_{r_{d}-1,d} = \{y \in \mathbb{N} \mid y \equiv 1, d + 2, \cdots, d+2^{r_{d}-2}\pmod{2d}\}. 
\]

We now state several lemmas that will
allow us to prove our main results. We first present a comparison theorem of Andrews \cite[Theorem 3]{andrews_partition_1971}.

\begin{theorem}[Andrews \cite{andrews_partition_1971}]\label{thm:andrews ST}
Let \(S = \{x_i\}_{i = 1}^{\infty}\) and \(T = \{y_i\}_{i = 1}^{\infty}\) be two strictly increasing sequences of positive integers such that \(y_1 = 1\) and \(x_i \geq y_i\) for all \(i\). Then for all \(n \geq 1\),
\[
    \rho(T\mbox{;}n) \geq \rho(S\mbox{;}n).
\]
\end{theorem}

We present the following modification of Theorem \ref{thm:andrews ST}, which will be extensively employed throughout the rest of this paper.

\begin{lemma}\label{lem:mod andrews}
    Let $S = \{x_{i}\}$ and $T = \{y_{i}\}$ be two  strictly increasing sequences of positive integers such that $y_{1} =a$, $a$ divides each $y_{i}$, and $x_{i} \geq y_{i}$ for all $i$. Then for all $n \geq 1$, \[\rho(T\mbox{;} n + h_{n}^{(a)} ) \geq \rho(S\mbox{;} n).\]
\end{lemma}
\begin{proof}[Proof.]

  We construct an injection $\varphi: X \to Y$ with $X$ and $Y$ being the sets of partitions counted by $\rho(S\mbox{;}n)$ and $\rho(T\mbox{;}n +h_{n}^{(a)})$ respectively. Suppose $\lambda  \in X$. Let $p_{i}$ and $q_{i}$ denote the multiplicity of $x_{i}$ (wrt $y_{i}$) occurs as a part of $\lambda$ (wrt $\varphi(\lambda)$). 
 
 We define an associated sum of differences for $\lambda$ by
\begin{equation}\label{eq:sum-differences}
    \alpha(\lambda)  := \sum_{i \geq 1}^{}p_{i}(x_{i} - y_{i}).
\end{equation}
We observe that $\alpha(\lambda)$ is non-negative since $x_{i} \geq y_{i}$ for all $i.$ Since $n = \sum_{i \geq 1} p_{i}y_{i} + \alpha$ and $ a$ divides $ \sum_{i \geq 1}p_{i}y_{i}$, we have that $a$ divides $\alpha + h_{n}^{(a)} $.

Using (\ref{eq:sum-differences}), we define 
\begin{equation*}
q_{i} = \begin{cases} 
p_{1} + \frac{h_{n}^{(a)}+ \alpha(\lambda)}{a}, i = 1\\
p_{i}, i \geq 2.
\end{cases}
\end{equation*}

Note
\begin{align*}
\sum_{i\geq 1}q_{i}y_{i} &= \left(p_{1} +  \frac{  \alpha(\lambda)+h_{n}^{(a)}}{a} \right)a+\sum_{i \geq 2}p_{i}y_{i} = p_{1}x_{1}  +  \sum_{i \geq 2}p_{i}x_{i} +h_{n}^{(a)} = n + h_{n}^{(a)}, 
\end{align*}
hence $\varphi$ is well defined. 

We now show that $\varphi$ is injective. Suppose that $\lambda, \lambda' \in X$, such that $\varphi(\lambda) = \varphi(\lambda')$. Let $p_{i}$ and $p_{i}'$ denote the multiplicity of $x_{i}$ occurring as parts of $\lambda$ and $\lambda'$ respectively. Observe from construction of $\varphi$ that we must have $p_{i} = p_{i}'$ for all $i \geq 2$. Note that this implies that 
\[
 \alpha(\lambda)= \sum_{i \geq 1}p_{i}(x_{i}-y_{i}) = \sum_{i \geq 1}p_{i}'(x_{i} - y_{i}) ' = \alpha(\lambda'),
\]
thus $p_{1} = p_{1}'$. Hence, we have $\lambda = \lambda',$ implying that $\varphi$ is injective. 

\end{proof}

We also employ the following lemmas of Duncan et al. \cite[Lemmas 2.4 and 2.5]{ourpaper} which allow us to use our modification of Alder's conjecture as described in Section \ref{sect:modalders}. 
\begin{lemma}[Duncan et al. \cite{ourpaper}]\label{lem:qstarlemma}
    For all $d, a \geq 1$ and $n \geq d + 2a$,
    \[ 
        \qda \geq \q{\lceil\frac{d}{a}\rceil}{1}{\lceil \frac{n}{a} \rceil}.
    \]
\end{lemma}

\begin{lemma}[Duncan et al. \cite{ourpaper}]\label{lem:Qidentitylemma}
    For $d,a,n\geq 1$ such that $a$ divides $d+3$,
    \begin{align*}
        \Qdash{d}{a}{an} &= \Qdash{\frac{d+3}{a}-3}{1}{n},
        \\
         \Qdashdash{d}{a}{an} &= \Qdashdash{\frac{d+3}{a}-3}{1}{n}.
\end{align*}
\end{lemma}

We now present a combined result of Andrews  \cite{andrews_partition_1971}, which will be employed in the proofs of Theorems \ref{thm:a=3} and \ref{thm:strengthen unc}.

\begin{theorem}[Andrews \cite{andrews_partition_1971}, 1971]\label{thm:andrews71}
For $d = 2^{r} - 1, r \geq 4$, and $ n \geq 1$, 
\[
\q{d}{1}{n} \geq  \rho(T_{r, d}\mbox{;} n). 
\]
\end{theorem}
\begin{proof}[Proof.]
The result follows by combining the proofs of \cite[Theorems 1 and 4]{andrews_partition_1971}. 
\end{proof}
We conclude this section with a result from Yee \cite{yee_alders_2008} which will be used in our modification of Alder's conjecture.

\begin{lemma}[Yee \cite{yee_alders_2008}, 2008]\label{lem:Ye08}
    For $d \geq 31, r \geq 1, d \neq 2^{r} -1$, and $n \geq 4d + 2^{r_{d}}$, \[\q{d}{1}{n} \geq \mathcal{G}_{d}^{(1)}(n) .\]
\end{lemma}
\begin{proof}[Proof.]
 Combine both \cite[Lemmas 2.2 and 2.7]{yee_alders_2008} to obtain the result. 
\end{proof}

\section{A modification of Alder's conjecture}\label{sect:modalders}

In this section, we use the work of Andrews \cite{andrews_partition_1971} and Yee \cite{yee_alders_2008} to prove a modification of Alder's conjecture. We  use this modification to give simple proofs of Theorems \ref{thm:a=3} and \ref{thm:strengthen unc}.  \begin{proposition}\label{prop:mod alders} 
For $d = 31$ or $d \geq 63$ and $n\geq  d+2$,
\[
q_{d}^{(1)}(n) \geq Q_{d-3}^{(1,-)}(n).
\]
\end{proposition}

We prove Proposition \ref{prop:mod alders} in two cases based on the form and size of $d$ and $n$. We will use throughout the paper the notation $\lambda^{x}$ to denote that $\lambda$ appears as a part $x$ times in a partition.

\begin{lemma}\label{lem:small m k-3}
For $ d+2 \leq n \leq 5d$ and $d \geq 31$, 
\[
 \q{d}{1}{n} \geq \Qdash{d-3}{1}{n} .
\]
\end{lemma}
\begin{proof}[Proof.]
 We define 
\[
S_{d} = \{x \in \mathbb{N} \mid x \equiv \pm 1 \pmod{d}\}\setminus{\{d-1\}}
\]
so that $\rho(S_{d}\mbox{;}n) = \Qdash{d -3}{1}{n}$. We prove Lemma \ref{lem:small m k-3} based on relating the size of the parts in $S_{d}$ and the size of $n$. 

We note that $q_{d}^{(1)}(n)$ is a weakly increasing function of $n$ since for any partition of $n$ counted by $q_{d}^{(1)}(n)$, we can add $1$ to its largest part to create a partition of $n +1$ counted by $q_{d}^{(1)}(n+1)$. In a similar fashion, $Q_{d-3}^{(1,-)}(n)$ is weakly increasing since for any partition of $n$ counted by $Q_{d-3}^{(1,-)}(n)$, we can adjoin $1$ as a part to create a partition of $n+1$ counted by $Q_{d-3}^{(1,-)}(n+1)$.

 Notice that for $d + 2 \leq n \leq 2d -2$ that we have $\q{d}{1}{d+2} = 2$ with partitions $(d +1, 1), (d +2)$. Note $\Qdash{d-3}{1}{2d - 2} = 2$ with partitions $( d+1, 1^{d-3}), ( 1^{2d -2}) $, hence the inequality holds.

We now consider the interval $2d - 1\leq n \leq 4d-1$. We notice that  $\q{d}{1}{2d -1} \geq 16$ since the partitions $(2d -1 ), ( 2d -1 - i, i)$ with $1 \leq i \leq 15$ are counted by $\q{d}{1}{2d-1}$ due to $d \geq 31$. Note that at $ n = 4d-1$,  we have $\Qdash{d-3}{1}{4d-1} = 12 $ partitions since there is one partition with largest part for each element in $\{4d-1 $, $3d+1, 3d-1\}$, two with largest part $2d+1$, three with largest part for each element in $\{2d -1, d +1 \}$, and one with largest part $1$. Hence for all $2d -1 \leq n \leq 4d - 1$ the inequality holds. 

We now verify for all $4d \leq n \leq 5d$ that the inequality holds. Notice that we have the lower bound $\q{d}{1}{4d} \geq 47$ since the partitions $(4d), (4d-i, i)$ with $1 \leq i \leq 46$ are counted by $\q{d}{1}{4d}$ due to $d \geq 31$. We observe that $\Qdash{d-3}{1}{5d} = 26$ since there is one partition with largest part for each element in $\{5d-1, 4d+1\}$, two partitions with largest part  $4d-1$, three partitions with largest part $3d+1$, four partitions with largest part $3d -1$, five partitions with largest part for each element in $\{2d+1, 2d-1\}$, four partitions with largest part $d+1$, and one partition with largest part $1$. Hence, we obtain $\q{d}{1}{n} \geq \Qdash{d-3}{1}{n}$ for $4d \leq n \leq 5d$. 

\end{proof}

\begin{lemma}\label{lem:large m, not power}
For $d = 31$ or $d \geq 63$  and $n \geq 4d + 2^{r_d}$, 
\[
q_{d}^{(1)}(n) \geq \Qdash{d-3}{1}{n}. 
\]
\end{lemma}
\begin{proof}[Proof.]
We prove Lemma \ref{lem:large m, not power} by showing the following inequalities,
\[
q_{d}^{(1)}(n) \geq \mathcal{G}_{d}^{(1)}(n) \geq \Qdash{d-3}{1}{n}. 
\]
In the case when $d \neq 2^{s} -1$, recall that from Lemma \ref{lem:Ye08} that $\q{d}{1}{n} \geq \mathcal{G}_{d}^{(1)}(n)$. From (\ref{eq:G function}), $\mathcal{G}_{d}^{(1)}(n) \geq \rho(T_{r_{d}-1, d}\mbox{;}n) $.  Since $r_{d} \geq 6$, we have $\rho(T_{r_{d}-1,d}\mbox{;} n) \geq \rho(T_{5, d}\mbox{;}n)$ from Theorem \ref{thm:andrews ST}. 

In the case when $d = 2^{s} -1$ with $s \geq 5$, from Theorem \ref{thm:andrews71}, we have $q_{d}^{(1)}(n) \geq \rho(T_{5,d}\mbox{;} n)$. Hence, in both cases it suffices to show $\rho(T_{5, d}\mbox{;}n) \geq \Qdash{d-3}{1}{n}$ 

Let $S$ and $T$ denote the sets of partitions counted by $\Qdash{d-3}{1}{n}$ and $\rho(T_{5,d}\mbox{;}n)$ respectively. 
We set $x_{i} \in S_{d}$ and $y_{i} \in T_{5,d}$ to denote the associated $ith$ smallest element of $S_d$ and $T_{5, d}$. Using Table \ref{tab:table 1}, note that the only $i$ where $x_{i} < y_{i}$ is when $i = 2$.


 We will construct an injection $\varphi: S \to T$. Let $\lambda \vdash n$ be an element in $S$ and $p_{i}$ and $q_{i}$ denote the number of times $x_{i}$ (wrt $y_{i}$) occurs as a part of $\lambda $ (wrt $\varphi(\lambda)$). Set
\[
\alpha:= \sum_{i \neq 2}p_{i}(x_{i} - y_{i})
\]
to be the difference sum.

Let $S_{1}$ denote the subset of $S$ where the partitions satisfy the constraint $p_{1} + \alpha \geq p_{2}$. Note that if $\lambda \vdash n$ has $p_{2} = 1$ that $\lambda \in S_{1}$ since $n \geq d +2$. We define the function $\varphi_{1} : S_{1} \to T$ as follows:

\textbf{I}: $p_{1} + \alpha \geq p_{2}$. We set
$$q_i = \begin{cases}
-p_2 + p_1 + \alpha, i = 1\\
p_i, i \geq 2\\
\end{cases}$$
\\
\begin{table}[]
\begin{center}
\caption{Values of $x_i \in S_d, y_i \in T_{5, d}$ for $i = 10\alpha + \bar{i}$ where $\bar{i}$ and $\alpha$ are respectively the remainder and quotient from Euclidean division of $i$ by $10$. \label{tab:table 1}}%
\begin{tabular}{l|l|l}
\(i\)                                                                                                     & \(x_i \)                               & \(y_i \)                           \\
\hline
\(1\)  & \(1\)   & \( 1 \) \\ \hline
 \(10\alpha + 1: \mbox{ ($i \neq 1$})\)  & \(\left(5\alpha+ 1\right)d - 1\)   & \( 4d\alpha + 1 \)  \\ \hline
 \(10\alpha + 2\) & \(\left(5\alpha+ 1\right)d + 1 \) &  \( 4d\alpha + d + 2\) \\ \hline
 \(10\alpha + 3\) & \(\left(5\alpha+2\right)d - 1\) & \(4d \alpha + d + 4\) \\ \hline
  \(10\alpha + 4\) & \(\left(5\alpha+2\right)d + 1 \) &  \( 4d\alpha + d + 8\) \\ \hline
 \(10\alpha+5\) & \(\left(5\alpha+3\right)d - 1\) & \(4d \alpha +d + 16\)
  \\ \hline \(10\alpha + 6\)  & \(\left(5\alpha+3\right)d + 1\)   & \( 4d\alpha + 2d + 1 \)  \\ \hline \(10\alpha + 7\)  & \(\left(5\alpha+4\right)d - 1\)   & \( 4d\alpha + 3d + 2 \)  \\ \hline \(10\alpha + 8\)  & \(\left(5\alpha+4\right)d + 1\)   & \( 4d\alpha + 3d + 4 \) \\ \hline \(10\alpha + 9\)  & \(\left(5\alpha+5\right)d - 1\)   & \( 4d\alpha + 3d + 8\) \\ \hline \(10\alpha: \alpha > 0 \)  & \(\left(5\alpha\right)d + 1\)   & \( 4d\alpha - d + 16\)  \\ \hline
 \end{tabular}
 \end{center}
 \end{table}
Observed that if $ \lambda \in S_{1}$ that we have $\varphi_{1}(\lambda) \vdash n$ since 
\begin{align*}
\sum_{i \geq 1}q_{i}y_{i} &= -p_{2} + p_{1} + \alpha + p_{2}(d+2) + \sum_{i \geq 3}p_{i}y_{i} \\
&= p_{1} + p_{2}(d+1) + \sum_{i \geq 3}p_{i}x_{i}  = n.
\end{align*}
Hence $\varphi_{1}$ is well-defined.

Let $S_{2}$ denote the set of partitions of $S$ with $p_{2} > p_{1} + \alpha$. We define  $S_{(2, \beta)} \subset S_{2}$ to be the set of partitions which additionally satisfy  $p_{1} + p_{6} = \beta(d-2) + \bar{p}$ with $\bar{p} \in \{0,\cdots, d-3\}$ and $\beta \in \mathbb{Z}_{\geq 0}$. Observe by construction that $S_{2}$ is the disjoint union $\bigcup_{\beta \in \mathbb{Z}_{\geq 0}}S_{(2,\beta)}$.

We define for fixed $\beta  \geq 0$, 
\[
\epsilon = \epsilon(\lambda) = \begin{cases} 0\ \text{if } p_{2} \text{ is even} \\
1\ \text{if } p_{2} \text{ is odd}. 

\end{cases} 
\]
 For each $\beta$, we define the function $\varphi_{(2,\beta)}: S_{(2, \beta)} \to T$  as follows:

\textbf{II}: $p_{1} + \alpha < p_{2}$ and $\lambda \in S_{(2,\beta)}$. We set \[
q_{i} = \begin{cases}
\frac{p_{2}-2\beta -3\epsilon}{2} + \alpha + p_{1} - 2\beta, i =1 \\
2\beta + \epsilon, i = 2 \\
p_{i}, 3 \leq i \leq 5\\
\frac{p_{2} -2 \beta - \epsilon}{2} + p_{6}, i = 6\\
p_{i}, i \geq 7. 
\end{cases} 
\]
We now show that $\varphi_{2,\beta}$ is well-defined. If $\lambda \in S_{2,\beta}$, notice 
\begin{align*}
    \sum_{i \geq 1}q_{i}y_{i} &= \left(\frac{p_{2}-2\beta -3\epsilon}{2} + \alpha + p_{1} - 2\beta  \right) + (2\beta + \epsilon )(d+2) \\
    +&\left(\frac{p_{2} -2 \beta - \epsilon}{2} + p_{6}\right)(2d+1)+ \sum_{ i \neq 1,2,6}p_{i}y_{i}\\
    &= p_{1} + \frac{p_{2}-2\beta -3\epsilon}{2} -2\beta + (2\beta + \epsilon)(d+2)\\
    +& \left(\frac{p_{2} - 2\beta - \epsilon}{2}\right)(2d+1) + \sum_{i \geq 3}p_{i}x_{i}\\
    &=p_{1} + p_{2} - 2\beta -\epsilon + (2\beta + \epsilon)(d+1) +d(p_{2} - 2\beta - \epsilon) + \sum_{i \geq 3}p_{i}x_{i}\\
    &= p_{1} + p_{2}(d+1) + \sum_{i \geq 3}p_{i}x_{i} = n. 
   \end{align*}

Hence $\varphi_{2,\beta}(\lambda) \vdash n$. 

In the case when $\epsilon = 0$, note $p_{2} - 2\beta \geq 0$ since $ p_{2} -\frac{2p_{2}}{d-2} \geq 0$ if $d \geq 4$. Additionally, $\alpha + p_{1} - 2\beta \geq 0 $ if $d \geq 4$, thus $q_{1}, q_{6} \geq 0$.

In the case when $\epsilon = 1$, note if $p_{2} = 3$ we must have $\beta = 0$, since $p_2 > p_{1} + \alpha \geq p_{1} + p_{6} \geq \beta(d-2) $, $d \geq 31$, and $\lambda \in \varphi_{2,\beta}$. For $p_{2} \geq 5$, note $p_{2} \geq 3 + \frac{2p_{2}}{d-2}$ if $d \geq 7$, implying $p_2 - 3 - 2\beta \geq 0$. Observe from the definition of $\beta$ that $\alpha + p_{1} - 2\beta \geq 0 $ for $d \geq 4$. Thus $q_{1}, q_{6} \geq 0$.  Hence, we obtain that $\varphi_{2, \beta}$ is well defined.

We define $\varphi: S \to T$ to be the function defined piecewise from  $\varphi_{1}, \varphi_{(2, \beta)}$ as above. In order to show that $\varphi$ is injective, it suffices to show that $\varphi_{1}, \varphi_{(2,\beta)}$ are injective and that the images of distinct cases are disjoint. 

Injectivity of $\varphi_{1}$ follows in the same manner as the function $\varphi$ present in Lemma \ref{lem:mod andrews}.  We now show for fixed $\beta$ that $\varphi_{(2, \beta)}$ is injective on its domain $S_{(2, \beta)}$. Suppose that $\lambda, \lambda' \in S_{(2, \beta)}$ are partitions such that $\varphi_{(2, \beta)}(\lambda) = \varphi_{(2,\beta)}(\lambda')$. Let $p_{i}$ and $p_{i}'$ denote the multiplicity numbers of $x_{i}$ occurring as a part of $\lambda$ and $\lambda'$ respectively. Similarly, let  $\bar{p},\bar{p}' $ denote the remainders when $p_{1} + p_{6}, p_{1}' + p_{6}'$ are divided by $d-2$ respectively. 

Observe from the definition of $\varphi_{(2,\beta)}$ that we may immediately have $p_{i} = p_{i}'$ for $i \neq 1,2,6$. Similarly, we must have $\epsilon(\lambda) = \epsilon(\lambda')$ otherwise $q_{2}$ and $q_{2}'$ will have opposite parity. Hence, we will assume that $p_{i}$ for $i \neq 1,2,6$ and $\epsilon(\lambda)$, $\epsilon(\lambda')$ are zero.  We obtain from the definition of $\varphi_{(2, \beta)}$ and using that $\beta$ is fixed the following system of equations 
\begin{align*}
\frac{p_{2}}{2} + dp_{6} + p_{1}  &=\frac{p_{2}'}{2} + dp_{6}'+ p_{1}'\\
\frac{p_{2}}{2} + p_{6} &= \frac{p_{2}'}{2} + p_{6}'\\
p_{1} + (d-1)p_{6} &= p_{1}' + (d-1)p_{6}'. 
\end{align*} 

Assume without loss of generality that $p_{1} \geq p_{1}'$. Observe that since $\lambda, \lambda' \in S_{(2, \beta)}$, we have $(p_{1} - p_{1}') + (p_{6} - p_{6}') = \bar{p} - \bar{p}' < d-2$. We note that this yields $\mid \bar{p}- \bar{p}'\mid < d -2$ since $0 \leq \bar{p}, \bar{p}' < d-2$. Using this and the third equation above, we have 
\begin{equation}\label{eq:barp}
(p_{1} - p_{1}') = (\bar{p}-\bar{p}') + (p_{6}' - p_{6} ) = (d-1)(p_{6}' - p_{6}). 
\end{equation} 
From (\ref{eq:barp}), we have that $\bar{p} = \bar{p}'$ since $d-2 \nmid (\bar{p} - \bar{p}')$ if $\bar{p} - \bar{p}' \neq 0$. This implies that $(d-2)(p_{6}' - p_{6}) = 0$, which yields that $p_{6} = p_{6}'$ since $d \geq 31$. Via the three equations above, we obtain that $p_{1} = p_{1}'$ and $p_{2} = p_{2}'$, thus $\lambda = \lambda'$. Hence $\varphi_{2,\beta}$ is an injection.

We now show that images of the subfunctions forming $\varphi$ are disjoint. We observe from construction that if $\beta \neq \beta'$ that $\igm \varphi_{(2,\beta)} \cap \igm \varphi_{(2, \beta')} = \varnothing $. 
Hence it suffices to show $\igm \varphi_{1} \cap \igm \varphi_{(2, \beta)} = \varnothing$ for all $\beta \in \mathbb{Z}_{\geq 0}$. Let $\lambda \in S_{1}$, $\lambda' \in S_{(2,\beta)}$, and suppose that $\varphi_{1}( \lambda) = \varphi_{(2,\beta)}(\lambda')$. We can assume from the construction of the function $\varphi_{2,\beta}$ that $p_{i} = p_{i}' = 0$ for all $i \neq 1,2,6$.

 We observe that we obtain the system of equations 
\begin{align*}
p_{1} + d p_{6} - p_{2} &= p_{1}' + d p_{6}'+\frac{p_{2}' - 2\beta - 3\epsilon}{2} - 2 \beta \\
p_{2} &= 2\beta + \epsilon\\
p_{6} &= p_{6}' + \frac{p_{2}'- 2\beta - \epsilon}{2}. 
\end{align*}
Using these three equations and $p_{1}' < p_{2}'$, we obtain
\begin{equation}\label{eq:p_{1}-rep}
p_{1} = -d(p_{6} - p_{6}') + p_{1}' + \frac{p_{2}' -2\beta -3\epsilon}{2} + \epsilon < \frac{-d +1}{2}(p_{2}'- 2\beta - \epsilon)  + p_{2}' . 
\end{equation}
Note from (\ref{eq:p_{1}-rep}) that it suffices to show 
\begin{equation}\label{eq:reduced-inequality}
    \frac{-d+1}{2}(p_{2}' - 2\beta - \epsilon)+p_{2}' \leq 0. 
\end{equation}

From (\ref{eq:reduced-inequality}), $(d-1)\beta \leq \frac{d-1}{d-2}p_{2}'$ and $\epsilon \leq 1$, it suffices to show that $\frac{d-1}{2}(p_{2}' - 1) - \frac{d-1}{d-2}p_{2}' \geq p_{2}'. $ We notice that $p_{2}' - 1 \geq \frac{p_{2}'}{2}$ since $p_{2}' \geq 2$. Hence it suffices to show that 
\begin{align*}
  \frac{d-1}{4} \geq 1 + \frac{d-1}{d-2},
\end{align*}
which is true for $d \geq 10$. This yields that $p_{1} < 0$ which is a contradiction. 
\end{proof}

\begin{proof}[Proof of Proposition \ref{prop:mod alders}.]
Combine Lemmas \ref{lem:small m k-3} and \ref{lem:large m, not power} to obtain the result. 
\end{proof}

\section{Proofs of Theorems \ref{thm:a=3} and \ref{thm:strengthen unc}}\label{sect:a=3 and uncstrengthgenkp}

In this section, we provide proofs of Theorems \ref{thm:a=3} and \ref{thm:strengthen unc} using the lemmas established in Section \ref{sect:prelims} as well as Proposition \ref{prop:mod alders}.

\subsection{Proof of Theorem \ref{thm:a=3}}

We first prove Theorem \ref{thm:a=3} for $d = 1$ by employing the Glaisher bijection \cite{glaisher_bijection_1883}. 
\begin{proposition}\label{prop:a =3 d=1} For $n \geq 1$ and $d = 1$, 
\[
\Delta_{d}^{(3,-)}(n) \geq 0. 
\]
\end{proposition}
\begin{proof}[Proof.]
Let $S$ and $T$ denote the sets of partitions counted by $\Qdash{1}{3}{n}$ and  $\q{1}{3}{n}$ respectively.  Observe that the parts of partitions in $S$ are congruent to $\pm 3\pmod{4}$ and strictly greater than $1$. Thus we have that each part is odd and is greater than or equal to $3$. Suppose $\lambda \in S$ and let $\lambda_{i}$ denote its $ith$ largest part.

We define an injection $\varphi: S \to T$ by modifying the Glaisher bijection \cite{glaisher_bijection_1883}. Let $p_{i}$  denote the multiplicity of $\lambda_{i} $ as a part of $\lambda$. We write $p_{i} = 2^{a_{1}(i)} + \cdots + 2^{a_{j}(i)}$ (with $a_{1}(i) < a_{2}(i) < \cdots < a_{j}(i)$) in its binary expansion.  Define $\varphi$ part-wise by
\begin{equation*}
  \varphi(\lambda_{i})= (2^{a_{1}(i)}\lambda_{i}, \cdots, 2^{a_{j}(i)}\lambda_{i}).
\end{equation*}
We observe that $\varphi$ replaces $\lambda_{i}$ with the distinct parts $2^{a_{1}(i)}\lambda_{i},\cdots, 2^{a_{j}(i)}\lambda_{i}$. Since $\lambda_{i} \geq 3$, the parts of $\varphi(\lambda)$ are greater than or equal to $3$. Note if $\lambda_{i} \neq \lambda_{j}$ that $2^{a}\lambda_{i} \neq 2^{b}\lambda_{j}$ for any positive integers $a,b$. Hence, the parts of $\varphi(\lambda)$ are distinct. Observe that $\varphi$ is injective by the same reasoning in \cite{euler_partition_identity}.
\end{proof}
We now provide a proof of the $d = 2$ case for Theorem \ref{thm:a=3}. 
\begin{proposition}
For $n \geq 1$ and $d = 2$, 
\[
\Delta_{d}^{(3,-)}(n) \geq 0. 
\]
\end{proposition}
\begin{proof}[Proof.]
Recall that the second Rogers-Ramanujan identity yields that for all $n \geq 1$, 
\begin{equation}\label{eq:2nd-rogers}
\sum_{n=0}^{\infty}q_{2}^{(2)}(n)q^{n} =  \sum_{n = 0}^{\infty} \frac{q^{n(n+1)}}{(q\mbox{;} q)_{n}} = \frac{1}{(q^{2}\mbox{;} q^{5})_{\infty}(q^{3}\mbox{;} q^{5})_{\infty}} = \sum_{n=0}^{\infty}Q_{2}^{(2)}(n)q^n.  
\end{equation}
Multiplying by the factor $(1- q^{2}) $ on both sides of (\ref{eq:2nd-rogers}) yields 
\begin{align*}
(1-q^{3})\sum_{n=0}^{\infty}q_{2}^{(2)}(n)q^{n} &= \frac{(1-q^{2})}{(q^{2}\mbox{;}q^{5})_{\infty}(q^{3}\mbox{;}q^{5})_{\infty}}=  \sum_{n=0}^{\infty}\Qdash{2}{3}{n}q^n. 
\end{align*}
Note that when $n= 1,2$, it's clear that $q_{2}^{(3)}(n) = \Qdash{2}{3}{n} = 0$. By setting $m = n+2$, it suffices to show for $m \geq 3$ the inequality,
\[
q_{2}^{(2)}(m) - q_{2}^{(2)}(m -2) \leq q_{2}^{(3)}(m). 
\]
Let $q_{2}^{(2)}(m)^{*}$ denote the set of partitions of $m$ counted by $q_{2}^{(2)}(m)$ with the additional property that they contain a part of size $2$. Note that $q_{2}^{(2)}(m) - q_{2}^{(2)}(m)^{*} = q_{2}^{(3)}(m)$ by construction. 
Thus it suffices to show that $q_{2}^{(2)}(m)^{*} \leq q_{2}^{(2)}(m-2)$.

Let $X$ and $Y$ denote the set of partitions counted by $q_{2}^{(2)}(m)^{*}$ and  $q_{2}^{(2)}(m-2)$ respectively. We also let $\lambda = (\lambda_{1}, \cdots,\lambda_{i-1}, 2) \vdash m$ be a partition counted by $q_{2}^{(2)}(m)^{*}$. Assuming that $X$ is non-empty, we define the function $\varphi: X \to Y$,\
\[
 \varphi(\lambda) =  (\lambda_{1}, \cdots, \lambda_{i-1}).
\]
It is clear from construction that $\varphi(\lambda) \in Y$. We now show that $\varphi$ is injective. Note that if $\lambda, \lambda' \in X$ have different lengths, then $\varphi(\lambda) \neq \varphi(\lambda')$ since $\varphi$ subtracts the length of the partitions by $1$. Hence we may assume that $\lambda, \lambda'$ have the same length. Since $\varphi$ only removes the last part of partitions of $X$, we immediately must have $\lambda = \lambda'$ if $\varphi(\lambda) = \varphi(\lambda')$. This yields  $q_{2}^{(2)}(m)^{*} \leq q_{2}^{(2)}(m-2)$, which completes the proof.
\end{proof}
\begin{proof}[Proof of Theorem \ref{thm:a=3}.] Let $91 \leq d \leq 93$ and $d \geq 187$. The case when $n = 1,2$ is trivial since $\q{d}{3}{n} = \Qdash{d}{3}{n} = 0$. We observe  for $d \geq 1$ and $3 \leq n \leq d + 5$ that
\[
\Delta_{d}^{(3, -)}(n)  \geq 0
\]
since $3$ is the only possible part of partition counted by $\Qdash{d}{3}{n}$ in the interval.  Thus $\Qdash{d}{3}{n} \leq 1 \leq q_{d}^{(3)}(n)$. In the case when $n = d + 6$, we have $q_{d}^{(3)}(d+6) = 2 \geq \Qdash{d}{3}{d+6}$. Now let $n \geq d +7$. To prove Theorem \ref{thm:a=3} for this range of $d$ and $n$, it suffices to show the inequality chain 
\begin{equation}\label{eq:a=3 inequality}
       q_{d}^{(3)}(n) \geq \q{\frac{d + h_{d}^{(3)}}{3}}{1}{\frac{n + h_{n}^{(3)}}{3}} \geq \Qdash{\frac{d+h_{d}^{(3)}}{3} - 3}{1}{\frac{n + h_{n}^{(3)}}{3}} = \Qdash{d + h_{d}^{(3)} -3}{3}{n + h_{n}^{(3)}} \geq \Qdash{d}{3}{n}. 
\end{equation}
The first inequality in (\ref{eq:a=3 inequality}) is justified by Lemma \ref{lem:qstarlemma} for $n \geq d + 6$. For $d$ such that $d + h_{d}^{(3)} = 93$ or $d + h_{d}^{(3)} \geq 189$, the second inequality follows from Proposition \ref{prop:mod alders}. The equality follows from applying Lemma \ref{lem:Qidentitylemma}. The final inequality follows by applying Lemma \ref{lem:mod andrews} with $\rho(T\mbox{;} n + h_{n}^{(3)} ) = \Qdash{d + h_{d}^{(3)} - 3}{3}{n + h_{n}^{(3)}}$ and $\rho(S\mbox{;} n) = \Qdash{d}{3}{n} $. Note that the application of Lemma \ref{lem:mod andrews} is justified since $h_{d}^{(3)} \leq 3$.

\end{proof} 

\subsection{The Proof of Theorem \ref{thm:strengthen unc}}
\begin{proof}[Proof of Theorem \ref{thm:strengthen unc}.]

 The case when $1 \leq n \leq a - 1$ is trivial since  $q_{d}^{(a)}(n) = \Qdash{d}{a}{n} = 0$. We observe for $d + h_{d}^{(a)} \geq 31a$ and $a \leq n \leq d +(a+2)$ that $a$ is the only possible part of a partition counted by $\Qdash{d}{a}{n}$, hence $\Qdash{d}{a}{n} \leq 1 \leq q_{d}^{(a)}(n)$. Note for $d + (3+a) \leq n \leq d + 2a -1$ that the only potential partitions counted by $\Qdash{d}{a}{n}$ are $(a^{\lambda}), (d+ 3+a)$ with $\lambda$ such that $d + (3+a) \leq \lambda a \leq d+ 2a $. Hence $\Qdash{d}{a}{n} \leq 2$. Note that $\Qdash{d}{a}{n} = 2$ when $n = d+3+a$ and $a$ divides $d+3+a$. Observe that this happens only when $d \equiv -3\pmod{a}$. Clearly $\qda \geq 1$ on this interval. Thus for $d + h_{d}^{(a)} \geq 31a$, $1 \leq n \leq d+2a$, $n \neq d+3+a$ when $d \equiv -3\pmod{a}$ ,
\[
\Delta_{d}^{(a,-)}(n) \geq 0. 
\]
Now let $n \geq d + 2a$. In order to prove Theorem \ref{thm:strengthen unc}, we derive the following inequality chain 
\begin{equation}\label{eq:unc strengthen inequality}
       \qda \geq \q{\frac{d + h_{d}^{(a)} }{a}}{1}{\frac{n + h_{n}^{(a)}}{a}} \geq \Qdash{\frac{d+h_{d}^{(a)} }{a} - 3}{1}{\frac{n + h_{n}^{(a)} }{a}} = \Qdash{d + h_{d}^{(a)} -3}{a}{n + h_{n}^{(a)} } \geq \Qdash{d}{a}{n}. 
\end{equation}

We utilize the same argument present in our proof of Theorem \ref{thm:a=3}. The first inequality in (\ref{eq:unc strengthen inequality}) is justified by Lemma \ref{lem:qstarlemma} for $n \geq d + 2a$. For $d$ such that $d + h_{d}^{(a)} = 31a$ or $d + h_{d}^{(a)}  \geq 63a$, the second inequality follows from Proposition \ref{prop:mod alders}. The equality is a result of applying Lemma \ref{lem:Qidentitylemma}. The final inequality follows by applying Lemma \ref{lem:mod andrews} with $\rho(T\mbox{;}n+h_{n}^{(a)} ) = Q_{d +h_{d}^{(a)} -3}^{(a,-)}(n+ h_{n}^{(a)} )$ and $\rho(S\mbox{;}n) = Q_{d}^{(a,-)}(n)$. The application of Lemma \ref{lem:mod andrews} holds since $h_{d}^{(a)} \leq 3$.

We end by remarking that  $h_{d}^{(4)}$ is always less than or equal to 3. Thus, for $n \geq 1$ and $\lceil\frac{d}{4} \rceil = 31$ or $\lceil\frac{d}{4} \rceil \geq 63$, and when $d \equiv 1\pmod{4}$ that $ n\neq d +7$, 
\[
\q{d}{4}{n} \geq \Qdash{d}{4}{n}.
\]
\end{proof} 
\section{On the generalized Kang-Park conjecture}\label{sect: generalized KP exp}
In this section, we provide a proof of Theorem \ref{thm:genkp unc}. This allows us to give an extension of \cite[Theorem 1.6]{ourpaper}. 

We define $r_{d, a}$ to be the largest non-negative integer $r$ such that $2^{r_{d, a}} -1 \leq \frac{d+h_{d}^{(a)}}{a} $. Recall that  $\mathcal{G}_{\frac{d + h_{d}^{(a)} }{a}}^{(1)}(\frac{n + h_{n}^{(a)} }{a})$ counts the number partitions of $\frac{n + h_{n}^{(a)} }{a}$ with the set of parts
\begin{align*}
\{\lambda_{i} \equiv 1, \frac{d + h_{d}^{(a)} }{a} + 2, \cdots,  \frac{d + h_{d}^{(a)} }{a} +2^{r_{d, a}-2},\frac{d + h_{d}^{(a)} }{a} + 2^{r_{d, a}-1} \pmod{2(\frac{d + h_{d}^{(a)} }{a})}\},
\end{align*}
where parts congruent to $ \frac{d + h_{d}^{(a)} }{a} + 2^{r_{d,a}-1} \pmod{2(\frac{d + h_{d}^{(a)} }{a})}$ are distinct.

For the intermediate partition functions in this section, we define for fixed $\ell, a, d \geq 1$ the set $T^{\ell}_{a, d}$ to be
\begin{equation*}
    T^{\ell}_{a, d} := \{\lambda_{i} \equiv \ell, (d  +2)\ell, \cdots, (d +2^{a-1})\ell \pmod{2d\ell} \}.  
\end{equation*}
We will use $T^{\ell}_{a,d}$ to compare $\rho(T^{1}_{a, d}\mbox{;} n)$ and $\mathcal{G}^{(1)}_{d}(n)$ with other partition functions of similar forms and $\Qdashdash{d}{a}{n}$.

We will prove Theorem \ref{thm:genkp unc} in three cases based on the form and size of $d$ and $n$. In addressing all three cases, we use the following result.
\begin{lemma}\label{lem:Prereq1}
Let  $d$, $a$, and $n$ be positive integers such that $a\geq 5$, $d + h_{d}^{(a)}  \geq 2^{a+3}a-a$. Then, \begin{equation*}
    \rho(T_{a, d + h_{d}^{(a)}}^{a}; n + h_{n}^{(a)} ) \geq \Qdashdash{d}{a}{n}.
\end{equation*}\end{lemma}
\begin{proof}[Proof.]
Let $S_{d}$ denote the set of allowed parts of partitions counted by the functions $\Qdashdash{d}{a}{n}$. We set $x_{i} \in S_{d}$ and $y_{i} \in T_{a, d+h_d^{(a)}}^{a}$ to denote the $ith$ part of their respective sets in ascending order. To compare the values of $x_i, y_i$, we use that $d + h_{d}^{(a)} \leq d + a$, $d + h_{d}^{(a)} \geq a2^{a+3} - a$, and $a \leq \frac{d}{15}$. Note that by using Table \ref{tab:1}, we find that a sufficient bound for $d$ to ensure $x_{i} \geq y_{i}$ is  
\begin{equation}\label{eq:max-bounds}
d \geq 2^{a-1}a +2a - 3.
\end{equation}

\begin{table}[]
\begin{center}
 \caption{ Values of $x_i$, $y_i$ over $j$ for $a \geq 5$ where $j$ and $\alpha$ are integers such that $j \in \{2, 3, ..., a\}$ and $i = a\ell + j$ for some $\ell \in \mathbb{N}$. \label{tab:1}}
 \resizebox{\textwidth}{!}{%
 \begin{tabular}{l|l|l}
 \(i \)                                                                                                     & \(x_i \)                               & \(y_i \)                           \\
 \hline
 \(a\ell + 1: \mbox{ ($a\ell + 1$ is even)} \)     & \(\left( \frac{a}{2}\alpha +\frac{3}{2} \right)(d+3)-a\)  & \( 2(d+h_{d}^{(a)})\alpha + a\) \\
 \hline
 \(a\ell + 1: \mbox{ ($a\ell + 1$ is odd)} \)     & \(\left( \frac{a}{2}\alpha + 1 \right)(d+3)+a\)  & \( 2(d+h_{d}^{(a)})\alpha + a\) \\
 \hline
  \(a \ell + j: \mbox{ ($a \ell + j$ is even)}\)     & \(\left(\frac{a}{2}\alpha + \frac{j+2}{2} \right)(d+3)-a\)  & \( 2(d+h_{d}^{(a)})\alpha + (d+h_{d}^{(a)}) + 2^{j-1}a\) \\
 \hline
  \(a \ell + j: \mbox{ ($a \ell + j$ is odd)}\)     & \(\left(\frac{a}{2}\alpha + \frac{j+1}{2} \right)(d+3)+a\)  & \( 2(d+h_{d}^{(a)})\alpha + (d+h_{d}^{(a)}) + 2^{j-1}a\) \\
 \hline
 \end{tabular}}
 \end{center}
\end{table}

By our assumption on $d$, inequality (\ref{eq:max-bounds}) is satisfied, hence $x_{i} \geq y_{i}$. Since $x_i \geq y_i$ for all positive $i$, we can apply Lemma \ref{lem:mod andrews} with $\rho(T\mbox{;} n + h_{n}^{(a)} ) = \rho(T_{a, d+h_d^{(a)}}^{a}\mbox{;} n + h_{n}^{(a)} )$ and $\rho(S\mbox{;} n) = \Qdashdash{d}{a}{n}$ to obtain the result.

\end{proof}
Now we prove Theorem \ref{thm:genkp unc}. For brevity we denote  $m = \frac{n + h_{n}^{(a)} }{a}$ and $k = \frac{d +h_{d}^{(a)} }{a}$.
\begin{proof}[Proof of Theorem \ref{thm:genkp unc}.]
In the case when $1 \leq n \leq d  +2+a$, one can check that $\Qdashdash{d}{a}{n} = 0$, hence $\qda \geq \Qdashdash{d}{a}{n}$. When $d +3+a \leq n \leq d+2a$, note that $\Qdashdash{d}{a}{n} \leq 1 \leq \qda$, since $\Qdashdash{d}{a}{n} = 1$ only when $n = d+a+3$. 

We now consider the case when $n > d + 2a$. For these $n$ we consider the following inequality chain:
\begin{equation}\label{eq: KeyIneqChainLittle}\begin{split}
     \q{d}{a}{n} \geq \q{k}{1}{m} \geq \rho(T_{a, k}^{1}\mbox{;} m)  = \rho(T_{d+h_d, a}^{a}\mbox{;} n + h_{n}^{(a)} )  \geq \Qdashdash{d}{a}{n}.\end{split}
\end{equation}

We note that the first inequality in (\ref{eq: KeyIneqChainLittle}) holds by Lemma \ref{lem:qstarlemma}. The last inequality holds by Lemma \ref{lem:Prereq1}. The equality follows from the natural bijection of multiplying the parts by $a$. Therefore we reduce to showing that \begin{equation}\label{eq:ReducedProblem}  \q{k}{1}{m} \geq \rho(T_{a, k}^1\mbox{;} m)\end{equation} for the following three cases.

\textit{Case 1}: Let $m \geq 4k + 2^{r_{k}}$ and $k \neq 2^{s} - 1$. We prove Inequality (\ref{eq:ReducedProblem}) by showing \begin{equation}\label{eq:case-1-gen}
    \q{k}{1}{m} \geq \mathcal{G}_{k}^{(1)}(m) \geq \rho(T^1_{a, k}\mbox{;} m).
\end{equation}
The first inequality of (\ref{eq:case-1-gen}) is a result from Lemma \ref{lem:Ye08}. The second inequality follows from $k = \frac{d + h_{d}^{(a)} }{a} \geq 2^{a+3} -1 $ and that there are more parts allowed for partitions counted by $\mathcal{G}_{k}^{(1)}(m)$ than $\rho(T^1_{a, k}\mbox{;} m)$.

\textit{Case 2}: Let $\frac{n + h_{n}^{(a)}}{a} \geq \frac{4(d + h_{d}^{(a)})}{a} + 2^{r_{d,a}}$ and $\frac{d + h_{d}^{(a)}}{a} = 2^{r_{d,a}} - 1$. We prove (\ref{eq:ReducedProblem}) by showing \begin{equation}\label{eq:2nd-case}
    \q{k}{1}{m} \geq \rho(T_{r_{d, a}, k}^{1}\mbox{;} m) \geq \rho(T^1_{a, k}\mbox{;} m).
\end{equation}

We observe that it suffices to prove the inequality when $r_{d,a}= a + 3$ since this yields the minimum number of congruence classes for $\rho(T^1_{r_{d, a}, k}\mbox{;} m)$.  Observe that the first inequality of (\ref{eq:2nd-case}) follows from Theorem \ref{thm:andrews71}. The second inequality follows from $k = \frac{d + h_{d}^{(a)}}{a} \geq 2^{a+3} -1 $ and that there are more parts allowed for partitions counted by $\rho(T^1_{r_{d, a}, k}\mbox{;} m)$ than $\rho(T_{a, k}^1\mbox{;} m)$.

\textit{Case 3}:  Let $m \leq  4k + 2^{r_{k}}$. It suffices to show for $k \geq 2^{a+3} - 1$ and $ 1 \leq m \leq 5k+1$ the inequality (\ref{eq:ReducedProblem}).

 We again use that $q_{k}^{(1)}(m)$ is a weakly increasing function. We also observe that $\rho(T^1_{a,k}\mbox{;} m)$ is weakly increasing since we can add $1$ to be an additional part of a partition of $m$ to obtain a partition for $m +1$.

We prove (\ref{eq:ReducedProblem}) for the interval $1 \leq m \leq 2k+6$. Note that both functions on the interval $1 \leq m \leq k +1$ return one, hence we suppose that $k + 2\leq m \leq 2k$. We notice that $\rho(T^1_{a,k}\mbox{;} k + 2^{i}) = i +1$ for $1 \leq i \leq a -1$ and is constant on the intervals $2^{i} + k \leq m \leq k +2^{i+1} - 1$.  We note that $\q{k}{1}{2^{i} + k} \geq i +1$ for $i \geq 1$ since the partitions $(k+2^{i}), (k +2^{i} - \alpha, \alpha)$ with $\alpha \leq 2^{i-1}$ are counted by $q_{k}^{(1)}(2^{i} + k)$.

Note that in the interval $k + 2^{a-1} \leq m \leq 2k$ that we have $\rho(T^1_{a,k}\mbox{;} m) = a $. Hence, we consider the value of the functions on the interval $2k +1 \leq m \leq 2k + 6$. We have that the partitions $(2k+1-\alpha, \alpha)$ with $\alpha \leq \lfloor \frac{k+1}{2} \rfloor$ are counted by $\q{k}{1}{2k+1}$, yielding $\q{k}{1}{2k+1} \geq 1 + \lfloor \frac{k+1}{2} \rfloor \geq 2^{a+2}+1$. We have $\rho(T^1_{a,k}\mbox{;} 2k+6) = a+3$, from counting the number of partitions with largest parts $2k+1$, $k+2^{a-1}, k+2^{a-2}, ..., k+2, 1$. Therefore, for $2k+1 \leq m \leq 2k+6$, $\q{k}{1}{m} \geq \rho(T^1_{a,k}\mbox{;} m)$

We now show (\ref{eq:ReducedProblem}) for the interval $2k+6 \leq m \leq 3k+1$. Since $q_{k}^{(1)}(m)$ is weakly increasing, we will find a lower bound for $\q{k}{1}{2k+6}$. The partitions of $2k+6$ of the form  
   $ (2k+6 - i, i)$
where $i \in \{1, ..., \lfloor \frac{k}{2}+3 \rfloor\}$ and the trivial partition $(2k+6)$ are counted by $\q{k}{1}{2k+6}$. From this, we obtain the lower bound $\q{k}{1}{2k+6} \geq \lfloor \frac{k}{2} + 3\rfloor + 1$. By $k \geq 2^{a+3} - 1$, we find that\begin{align*}
    q_k^{(1)}(2k+6) \geq \lfloor \frac{2^{a+3} - 1}{2} + 3\rfloor + 1 = 2^{a+2} + 3.
\end{align*}

 Since $\rho(T^1_{a,k}\mbox{;} m)$ is weakly increasing, we will provide an upper bound for $\rho(T^1_{a,k}\mbox{;} 3k +1)$. The largest part that could be in a partition counted by $ \rho(T^1_{a,k}\mbox{;} 3k+1)$ is $2k+1$. We observe that $(2k+1, 1^k)$ is the only partition of $3k+1$ counted by $\rho(T\mbox{;} 3k+1)$ that includes $2k+1$ since $3k+1 - 2k - 1 = k$. 

We now consider partitions counted by $\rho(T^1_{a,k}\mbox{;} 3k+1)$ with largest part $k + 2^{\ell}$ where $\ell \in \{1, 2, ..., a-1\}$. For each $\ell$, we notice that $3k+1  - k - 2^{\ell} = 2k + 1 - 2^{\ell}$. Therefore, any partition whose largest part is of the form $k+2^{\ell}$ can have at most one other element of the form $k+2^w$ with $w \in \{1, ...., \ell\}$. Thus, for each fixed $\ell \in \{1, ..., a-1\}$, there are at most $(\ell + 1)$ partitions of $3k+1$ such that $k+2^{\ell}$ is the largest part. Finally, there is only one partition of $3k+1$ with largest part of 1. Therefore, we obtain that 
\begin{align*}\rho(T^{1}_{a, k}\mbox{;} 3k+1) \leq  \frac{(a+2)(a-1)}{2} + 1 + 1 = 1 + \frac{a(a+1)}{2}.
\end{align*}  Note that since $2^{a+2} + 3 \geq \frac{a(a+1)}{2} + 1$ for $a \geq 1$, we have that (\ref{eq:ReducedProblem}) holds for this interval. 

We now prove (\ref{eq:ReducedProblem}) for $3k+1 \leq m \leq 4k+1$. Since $q_{k}^{(1)}(m)$ is weakly increasing, we will find a lower bound for $\q{k}{1}{3k+1}$. The partitions of $3k + 1$ of the form $( 3k+1 - i, i)$ for $i \in \{1, ..., \lfloor \frac{2k+1}{2} \rfloor\}$ and the trivial partition $(3k+1)$ are counted by $\q{k}{1}{3k+1}$. Thus, we have that \begin{align*}
    q_k^{(1)}(3k+1) \geq \lfloor \frac{2k+1}{2}\rfloor + 1 \geq \lfloor \frac{2\cdot (2^{a+3} - 1)+1}{2}\rfloor + 1 = 2^{a+3}.
\end{align*} 

We now provide an upper bound for $\rho(T^1_{a,k}\mbox{;} m) $ in $3k+1 \leq m \leq 4k +1$ by obtaining an upper bound for $\rho(T^1_{a,k}\mbox{;} 4k+1)$. By repeating the same procedure by considering partitions with fixed largest part, there is one partition with largest part $4k + 1$, there are at most $a-1$ partitions with largest part of the form $3k + 2^{\ell}$ with $\ell \in \{1, 2, ..., a-1\}$, at most $a$ partitions with largest part $2k +1$, and for each $h \in \{1, \cdots, a-1\}$ at most $(h + 1)^2$ partitions with largest part $k + 2^{h}$ by considering the partitions with one, two, and three parts that are not equal to $1$. Finally, there is only one partition of $4k+1$ whose largest part is 1. Using this method, we have \begin{align*}
    \rho(T^1_{a,k}\mbox{;} 4k+1) \leq 1 + (a - 1) + a + (2^2 + 3^2 + ... + a^2) + 1 = 2a +  \frac{a(a+1)(2a+1)}{6}.
\end{align*}
Thus, we have that (\ref{eq:ReducedProblem}) holds for $3k + 1 \leq m \leq 4k + 1$ since $k \geq 2^{a+3} -1$.

Finally, we show that (\ref{eq:ReducedProblem}) holds for $4k+1 \leq m \leq 5k+1$.  Since $q_{k}^{(1)}(m)$ is weakly increasing, we will find a lower bound for $\q{k}{1}{4k+1}$. The partitions of $4k+1$ of the form $( 4k+1 - i, i)$ where $i \in \{1, ...., \lfloor \frac{3k+1}{2} \rfloor\}$ and the trivial partition $(4k+1)$ are counted by $q_k^{(1)}(4k+1)$. This yields the lower bound,
\begin{align*}
    q_{k}^{(1)}(4k+1) \geq \lfloor \frac{3k+1}{2}\rfloor + 1 \geq \lfloor \frac{3\cdot (2^{a+3} - 1)+1}{2}\rfloor + 1 = 3\cdot 2^{a+2}.
\end{align*}

We now provide an upper bound for $\rho(T^1_{a,k}\mbox{;} m) $ in $4k+1 \leq m \leq 5k+1$, which we'll do by obtaining an upper bound for $\rho(T^1_{a,k}\mbox{;} 5k+1)$. The only partition with largest part $4k+1$ is $(4k+1, 1^k)$. By repeating the same procedure used in the interval $2k +6 \leq m \leq 3k + 1$, we find that there are at most $a(a-1)$ partitions with largest part of the form $3k + 2^{\ell}$ with $\ell \in \{1, 2, ..., a-1\}$, at most $1 + \frac{a(a+1)}{2}$ partitions with largest part $2k + 1$ by considering various fixed second largest potential parts. We also find for each $h \in \{1, \cdots, a-1\}$ at most $(h +1)^{3}$ partitions whose largest part is $k + 2^{h}$ by considering that there are at most four parts that are not equal to $1$. Finally there is only one partition of $5k+1$ whose largest part is $1$. We obtain by adding the upper bound
\begin{align*}\rho(T^1_{a,k}\mbox{;} 5k+1) \leq 1 + a(a-1) + 1 + \frac{a(a+1)}{2}+  \frac{a^2(a+1)^2}{4}.
\end{align*} 
Because $k \geq 2^{a+3} - 1$, we have that (\ref{eq:ReducedProblem}) holds for this interval. 

To summarize, we obtain bounds for $\q{k}{1}{m}$ and $\rho(T^1_{a,k}\mbox{;} m)$ displayed in Table \ref{tab:3}. Notice that for $a\geq 5$, this implies that $\q{k}{1}{m} \geq \rho(T^1_{a,k}\mbox{;} m)$ for $2k+6 \leq m \leq 5k+1$. This completes the proof of (\ref{eq:ReducedProblem}) for $1 \leq m \leq 5k+1$

\begin{table}[]
\begin{center}
\caption{Bounds on $\q{k}{1}{m}$ and $\rho(T_{a,k}^{1}\mbox{;} m)$ for intervals of $m$.\label{tab:3}}
\resizebox{\textwidth}{!}{%
\begin{tabular}{l|l|l}
 \(m \)                                                                                                     & \(\q{k}{1}{m} \)                               & \(\rho(T^1_{a,k}\mbox{;} m) \)                           \\
\hline
 \(2k+6 \leq  m \leq 3k+1 \)     &  \(\geq 2^{a+2} + 3 \)  & \(\leq  1 + \frac{a(a+1)}{2}\) \\
\hline
 \(3k+1 \leq  m \leq 4k+1 \)     & \(\geq  2^{a+3}\)  &   \(\leq 2a + \frac{a(a+1)(2a+1)}{6}\) \\
\hline
  \(4k+1 \leq  m \leq 5k+1 \)     & \(\geq  3\cdot 2^{a+2}\)  &   \(\leq 2+ a(a-1)+  \frac{a(a+1)}{2}+  \frac{a^2(a+1)^2}{4}.\)\\
\hline
\end{tabular}}
\end{center}
\end{table}
\end{proof}

\section{A strengthening of the generalized Kang-Park conjecture}\label{sect: generalized KP con}

In this section, we provide a proof of Theorem \ref{thm:genkp con} by using Conjecture \ref{conj:aldermconj}. We also prove Conjecture \ref{conj:aldermconj} for $d+2\leq n \leq 5d$ to obtain that Theorem \ref{thm:genkp con} holds unconditionally for the specified range of $n$ in Theorem \ref{thm:genkp con}.

\begin{proof}[ Proof of conditional part of Theorem \ref{thm:genkp con}.] 

Via work done in the unconditional component of Theorem \ref{thm:genkp con}, we assume that $n \geq d+2a$. We employ the following inequality chain: 
\begin{align*}
   \qda \geq \q{\frac{d + h_{d}^{(a)} }{a}}{1}{\frac{n + h_{n}^{(a)} }{a}} \geq Q_{\frac{d + h_{d}^{(a)} }{a} - 4}^{(1, -)}(\frac{n + h_{n}^{(a)} }{a}) \\  = \Qdash{d  + h_{d}^{(a)} -a - 3}{a}{n + h_{n}^{(a)} } \geq \Qdash{d}{a}{n}. 
\end{align*}

 We note that the first inequality holds for $n \geq d + 2a$ by Lemma \ref{lem:qstarlemma}. The second inequality follows by assuming Conjecture \ref{conj:aldermconj}. The equality follows from the bijection by multiplying the  parts of partitions counted by $\Qdash{\frac{d + h_{d}^{(a)} }{a}  - 4}{1}{\frac{n + h_{n}^{(a)} }{a}}$ by $a$.
 We now show the last inequality.
 
 Let $X$ and $Y$ be the set of partitions counted by $\Qdash{d}{a}{n}$ and $\Qdash{d + h_{d}^{(a)} - 3 - a}{a}{n + h_{n}^{(a)} }$ respectively. Let $x_{i}$ and $y_{i}$ denote the allowed parts of partitions in $X$ and $Y$ with indexing with respect to increasing size. Note that for all $i\geq 1$ that we have $x_{i} \geq y_{i}$ since $h_{d}^{(a)} \leq a$. This allows us to apply Lemma \ref{lem:mod andrews}, yielding $\Qdash{d + h_{d}^{(a)} - a - 3}{a}{n + h_{n}^{(a)} } \geq \Qdash{d}{a}{n}$.  
\end{proof}
We now proceed with the unconditional part of Theorem \ref{thm:genkp con}. We first prove Theorem \ref{thm:genkp con} in the case when $1 \leq n \leq d + 2a$. 
\begin{lemma}\label{lem:very-small-n-genkpcon}
Let $a \geq 1$,  $\lceil \frac{d}{a} \rceil \geq 12$, and $1 \leq n \leq d + 3a$. If $d \not\equiv -3\pmod{a}$, we have
\[
\Delta_{d}^{(a,-)}(n) \geq 0.
\]
If $d \equiv -3\pmod{a}$ then $\Delta_{d}^{(a, -)}(n) \geq 0$ for all $1 \leq n \leq d + 3a$  except when $n = d + 3+a$. 
\end{lemma}
\begin{proof}[Proof.]
The case when $1 \leq n \leq a -1$ is trivial since $\qda = \Qdash{d}{a}{n} = 0$. Observe for $a \leq n \leq d +2+a$ that $\q{d}{a}{n} \geq  1\geq \Qdash{d}{a}{n}$. Note for $d + 3+a \leq n \leq d +2a$ that $\qda \geq 1 \geq \Qdash{d}{a}{n}$. If $d \not\equiv -3 \pmod{a}$,  we have $\Qdash{d}{a}{d+a+3} = 1$, hence $\q{d}{a}{d+a+3} \geq  1 =  \Qdash{d}{a}{d+a+3}$.

If $d \equiv -3 \pmod{a}$, we find that the only $n$ in the interval with $\Qdash{d}{a}{n} = 2$ is $n =  d + 3+a$. Thus for any $d+3+a \leq n \leq d+2a$ and $n \neq d+3+a$, we have $\qda \geq 1 \geq \Qdash{d}{a}{n}$. 

We now consider the case when $d +2a \leq n \leq d +3a$. Observe that $q_{d}^{(a)}(n) \geq 2$ with partitions $(n)$ and $(n-a,a)$ since $n \geq d + 2a$. Notice that $\Qdash{d}{a}{n} \leq 2$ since the only partitions in the interval are $a^{\lambda}$ and $(d+3+a, a)$. Thus $q_{d}^{(a)}(n) \geq 2 \geq \Qdash{d}{a}{n}$. 
\end{proof}
We now give an unconditional proof of the second statement of Theorem \ref{thm:genkp con} for the case where $\lceil \frac{d + 2a}{a} \rceil \leq \lceil \frac{n}{a} \rceil \leq 5\lceil \frac{d}{a} \rceil $.  For this we  use the following lemma.
\begin{lemma}\label{lem: Littlem}
For integers $d \geq 12$ and $d+2 \leq n \leq 5d$, 
   \[ \q{d}{1}{n} \geq \Qdash{d-4}{1}{n}.
   \]
\end{lemma}

\begin{proof}[Proof of Lemma \ref{lem: Littlem}.]
Recall that $ \q{d}{1}{n}$  is weakly increasing since for any partition of positive integer $n$ counted by the function $ \q{d}{1}{n}$, one can add 1 to the largest part of the partition to obtain a partition counted by $\q{d}{1}{n+1}$. Similarly, $Q_{d-4}^{(1,-)}(n)$ is weakly increasing since for any partition of $n$ counted by $Q_{d-4}^{(1,-)}(n)$, we can adjoin $1$ as a part to create a partition of $n+1$ counted by $Q_{d-4}^{(1,-)}(n+1)$. 

We begin by showing that $ \q{d}{1}{n} \geq \Qdash{d-4}{1}{n}$ for $d+2 \leq n \leq 2d-4$. Note that $\q{d}{1}{d+2} \geq 2$ since the partitions $(1, d+1)$ and $(d+2)$ are counted.  Note that $\Qdash{d-4}{1}{2d-4} = 2$ with associated partitions $(d, 1^{d - 4}), (1^{2d-4})$. Thus within the interval $d +2 \leq n \leq 2d -4$ the inequality holds. 

We now show $ \q{d}{1}{n} \geq \Qdash{d-4}{1}{n}$ for $2d-3 \leq n \leq 3d-5$.
We note that $\Qdash{d-4}{1}{3d-5} = 5$ with associated partitions $((2d-1), 1^{d-4}), ((2d-3), 1^{d-2}), (d^2, 1^{d-5}), (d, 1^{2d-5}), (1^{3d-5})$, hence it suffices to show that $q_{d}^{(1)}(n) \geq 5$ in this interval. We notice that the partitions of $2d-3$ in the set $\{(2d-3)\} \cup \{( 2d-3 - i, i): i \in \{1, 2,..., \lfloor \frac{d-3}{2} \rfloor \}\}$ are counted by $q_{d}^{(1)}(2d-3)$. Therefore, $q_{d}^{(1)}(2d-3) \geq 1 + \lfloor \frac{d-3}{2} \rfloor$. Since $d \geq 12$, we have that $q_{d}^{(1)}(2d-3) \geq 5$ as desired. 

 We now show that $q_d^{(1)}(n) \geq \Qdash{d-4}{1}{n}$ for $3d-4 \leq n \leq 4d-6$. We note that $\Qdash{d-4}{1}{4d - 6} = 11$ since there is one partition with largest part for each element in $\{3d-2, 3d-4\}$, two with largest part of $2d-1$, three with largest part for each element in $\{2d-3, d\}$, and one partition of largest part $1$. Hence, it suffices to show $q_{d}^{(1)}(n) \geq 11$ in the interval. We notice that the partitions of $3d-4$ in the set $\{(3d-4)\} \cup \{( 3d-4 - i, i): i \in \{1, 2,..., d-2 \}\}$ are counted by $q_{d}^{(1)}(3d-4)$. Therefore, $q_{d}^{(1)}(3d-4) \geq d-1 \geq 11$.

We now show that  $\q{d}{1}{n} \geq \Qdash{d-4}{1}{n}$ for $4d-5 \leq n \leq 5d-8$. We find that $\Qdash{d-4}{1}{5d-8} = 20$ since there is one partition with largest part for each element in $\{4d-3, 4d -5\}$, two with largest part for each element in $\{3d-2, 3d-4\}$,  five with largest part $2d-1$, four with largest part for each element in $\{2d-3, d\}$, and one with largest part $1$. Hence, it suffices to show $q_{d}^{(1)}(n) \geq 20$.  We notice the partitions of $4d-5$ in the set $\{(4d-5)\} \cup \{(i, 4d-5 - i): i \in \{1, 2,..., \lfloor \frac{3d-5}{2} \rfloor \}\}$ are counted by $q_{d}^{(1)}(4d-5)$. Note that $q_{d}^{(1)}(4d-5)$ also counts partitions of the form $(3d-6-j, j+d, 1)$ with $j \in \{1, 2, ..., \lfloor \frac{d-6}{2} \rfloor \}$ and partitions of form $(3d-8-\ell, \ell+1+d, 2)$ with $\ell \in \{1, 2, ..., \lfloor \frac{d-9}{2} \rfloor \}$. Therefore, $q_{d}^{(1)}(4d-5) \geq 1 + \lfloor \frac{3d-5}{2} \rfloor + \lfloor \frac{d-6}{2} \rfloor + \lfloor \frac{d-9}{2} \rfloor$. Since $d \geq 12$, we have the desired inequality. 

Finally, we prove that $\q{d}{1}{n} \geq \Qdash{d-4}{1}{n}$ for $5d-8 \leq n \leq 5d$. We find that $Q_{d-4}^{(1, -)}(5d) = 36$ since there is one partition with largest part for each element in $\{5d-4, 5d-6\}$, two with largest part for each element in $\{4d-3, 4d-5\}$, five with largest part for each element in $\{3d- 2, 3d-4\}$, eight with largest part $2d-1$, six with largest part $2d-3$, five with largest part $d$, and one with largest part $1$.  Hence it suffices to show that $q_{d}^{(1)}(n) \geq 36$. We notice that the partitions of $5d-8$ in the set $\{(5d-8)\} \cup \{(5d-8 - i,  i): i \in \{1, 2,..., 2d-4 \}\}$ are counted by $q_{d}^{(1)}(5d-8)$. In addition, $q_{d}^{(1)}(5d-8)$ counts the partitions of the form $(4d - 9 - j, d + j,1 )$ for $j \in \{1, \cdots, \lfloor \frac{2d -9}{2} \rfloor \}$. We also note $q_{d}^{(1)}(5d-8)$  counts the partitions of the form $(4d - 11 - \ell, d + 1+ \ell,2 )$ for $\ell \in \{1, \cdots, d-6 \}$ and  $(4d - 13 - r, d + 2+ r, 3)$ for $r \in \{1, \cdots, \lfloor \frac{2d-15}{2}\rfloor \}$. Therefore, $$q_{d}^{(1)}(5d-8) \geq 2d-4 + \lfloor \frac{2d -9}{2} \rfloor + d - 6 + \lfloor \frac{2d-15}{2}\rfloor \geq 36$$ using the assumption that $d \geq 12$.

\end{proof}

\begin{proof}[Proof of unconditional part of Theorem \ref{thm:genkp con}.]
The case $1 \leq n \leq d+3a $ with the additional condition $n \neq d+a -3$ when $d \equiv -3 \pmod{a}$ is addressed by Lemma \ref{lem:very-small-n-genkpcon}. 

For the case $n \geq d+3a$, we employ the same inequality chain as used in the conditional part of Theorem \ref{thm:genkp con}. Recall that
\begin{align*}
    \qda \geq \q{\frac{d + h_{d}^{(a)}}{a}}{1}{\frac{n + h_{n}^{(a)} }{a}}  \geq Q_{\frac{d + h_{d}^{(a)} }{a} - 4}^{(1, -)}(\frac{n + h_{n}^{(a)} }{a}) \\ = \Qdash{d + h_{d}^{(a)} -a - 3}{a}{n + h_{n}^{(a)} } \geq \Qdash{d}{a}{n}. 
\end{align*}

All inequalities and equalities except for the second one are justified by our work in the conditional component of Theorem \ref{thm:genkp con}. The second
inequality for the cases $\lceil \frac{d}{a} \rceil + 2 \leq \lceil \frac{n}{a} \rceil \leq  5\lceil \frac{d}{a}\rceil$ is resolved by Lemma \ref{lem: Littlem}.
\end{proof}

\section{Potential methods and future directions}\label{sect:asymptotics} 
In this section, we  describe potential methods in proving more cases of Conjectures \ref{conj:a=3}, \ref{conj:genkp} and extending the results of Theorem \ref{thm:strengthen unc}.

We observe by repeating the arguments presented in the proof of \cite[Theorem 1.9]{ourpaper} to obtain that for $a, n \geq 1$ and $d \geq 4$ such that $a < d+3$ and $\mbox{gcd}(a, d + 3) = 1$,
        \[
            \lim_{n \to \infty} \Daa = \lim_{n \to \infty} \qda \left(1 - \frac{\Qda}{\qda}\right) = +\infty.
        \]
     Observe Remark \ref{rem:strengthen-inequality} indicates that $\Daadash \geq \Daa$, implying our desired result.

We note from Theorem \ref{thm:a=3} that the remaining cases of Conjecture \ref{conj:a=3} are $3 \leq d \leq 90,\ 94 \leq d \leq 186$. The sub-cases when $d \geq 93$ are divisible by $3$ are addressed by \cite[Proposition 4.1]{ourpaper}. 

We  describe a potential method in resolving more cases of Conjecture \ref{conj:a=3} and extending Theorem \ref{thm:strengthen unc}. First, one should derive computationally effective asymptotic expressions for $q_{\lceil \frac{d}{a}\rceil}^{(1)}(\lceil \frac{n}{a}\rceil)$ and $\Qdash{\lceil \frac{d}{a}\rceil-3}{1}{\lceil \frac{n}{a}\rceil}$ by using the work of Alfes, et.al \cite{alfes_proof_2011}. One can then use these expressions to show $q_{\lceil \frac{d}{a}\rceil}^{(1)}(\lceil \frac{n}{a}\rceil) \geq \Qdash{\lceil \frac{d}{a}\rceil-3}{1}{\lceil \frac{n}{a}\rceil} $ for suitable $\lceil \frac{n}{a}\rceil$ and use (\ref{eq:a=3 inequality}) and (\ref{eq:unc strengthen inequality}) to prove more small $d$ cases of Conjecture \ref{conj:a=3} and extend Theorem \ref{thm:strengthen unc}.

From application of the work of Kang and Kim \cite{kang_asymptotics_2021}, it appears that this approach will be feasible for $\lceil \frac{d}{a}\rceil \geq 10$. In particular, one can apply \cite[Theorem 1.1]{kang_asymptotics_2021}.
\begin{theorem}[Kang-Kim \cite{kang_asymptotics_2021}, 2021]\label{thm:KangKimAssympt}
Let $a$, $d$, $m$, $m_1$, $m_2$ be integers such that $0 \leq m_1 < m_2 < m$ and $a, d \geq 1$. For $\alpha_d$ the
unique real root of $x^d + x - 1$ in the interval $(0, 1)$ and $$A_d = \frac{d}{2}\log^2\alpha_d + \sum_{r=1}^{\infty}\frac{\alpha_d^{rd}}{r^2},$$ we let $M_d = \lfloor \frac{\pi^2}{3A_d}\rfloor$. Then, it can be found that $$\lim_{n \to \infty} \qda - \rho(\{p \in \mathbb{N}: p \equiv m_1 \pmod{m} \mbox{ or } p \equiv m_2 \pmod{m}\};n) $$ $$=  \begin{cases}+\infty \mbox{, if $m > M_d$} \\ -\infty \mbox{, if $m \leq M_d$.} \\ 
\end{cases}$$
\end{theorem} Calculations of $\alpha_d, A_d, M_d$ suggest that for $d \geq 10$, Theorem \ref{thm:KangKimAssympt} implies that $q_{d}^{(1)}(n) - Q_{d-3}^{(1)}(n) \to \infty$  as $n \to \infty$. This is further corroborated by computational data of verifying the inequality for $d + 2 \leq n \leq 100000$.
\begin{table}[]
\begin{center}
\caption{ Values of $d$, $\alpha_d$, $A_d$, $M_d$ as defined by Kang and Kim \cite{kang_asymptotics_2021} over select $d \geq 3$. \label{tab:4}}

\begin{tabular}{l|l|l|l}
 \(d\)                                                                                                     & \(\alpha_d \)  (approximate)                              & \(A_d \) (approximate)  & \(M_d \)                        \\
\hline
3    &  0.682328 & 0.566433 & 5 \\
\hline
  4  & 0.724492  & 0.504981  & 6 \\
\hline
5    &  0.754878 & 0.459731 & 7\\
\hline
6  & 0.778090  & 0.424486  & 7 \\
\hline
7  & 0.796544  & 0.395966  & 8 \\
\hline
8  & 0.811652  & 0.372243  & 8 \\
\hline
9  & 0.824301  & 0.352090  & 9 \\
\hline
10  & 0.835079  & 0.334683  & 9 \\
\hline
11  & 0.844398  & 0.319446  & 10 \\
\hline
12  & 0.852551  & 0.305958  & 10 \\ \hline
20  & 0.893895  & 0.234874  & 14 \\ \hline
100  & 0.966584  & 0.091456  & 35 \\
\hline
\end{tabular}
\end{center}
\end{table}

We now consider the sub-cases of Conjecture \ref{conj:a=3} consisting of $3 \leq d \leq 27$. Unfortunately, Table \ref{tab:4} and computational data suggest that the proposed method fails for those values of $d$ since for $3 \leq  d \leq 27$, we have $\lceil \frac{d}{3}\rceil \leq M_{\lceil \frac{d}{3}\rceil}$. This implies by \cite[Theorem 1.1]{kang_asymptotics_2021}, that
\[
\lim_{n \to \infty} \left( q_{d}^{(1)}(n) - \Qdash{d-3}{1}{n}\right) = -\infty. 
\]
To address this problematic case, one presumably could use the explicit asymptotic expressions in Duncan, et al. \cite{ourpaper} for $q_{d}^{(3)}(n)$ and $Q_{d}^{(3,-)}(n)$ to find an $\Omega(d)$ such that $n > \Omega(d)$,
\[
\Delta_{d}^{(3,-)}(n) \geq 0
\]
to address the values of $d> 3$ such that $\mathrm{gcd}(d,3) = 1$.
One presumably could then employ a finite computation to show $\Delta_{d}^{(3,-)}(n) \geq 0$ for $n \leq \Omega(d)$ for these values of $d$.

Remarkably, Armstrong, et.al \cite{armstrong_alder_type_2022} have extended Theorems \ref{thm:genkp unc} and \ref{thm:genkp con} with a linear lower bound on $d$ by proving a vast generalization of Proposition \ref{prop:mod alders}. In particular, they have shown that Theorem \ref{thm:genkp unc} holds for all $a,n \geq 1$ and $\lceil \frac{d}{a}\rceil \geq 105$. However, their methods do not seem applicable to reduce Conjecture \ref{conj:genkp} to a finite computation. Such a reduction, unfortunately appears out of reach of current combinatorial and analytic methods since Conjecture \ref{conj:aldermconj} fails for $1 \leq d \leq 11$ via Theorem \ref{thm:KangKimAssympt}.

\section*{Acknowledgments}
We are immensely thankful for Holly Swisher in many useful conversations, comments, and deeply encouraging feedback regarding this paper. We also thank Robert J. Lemke Oliver for giving us his code which greatly aided our calculations of $q_{d}^{(1)}(n)$ and $Q_{d}^{(1)}(n)$ for large $n$. We also thank the anonymous reviewer for numerous helpful comments and suggestions which greatly improved the exposition of this paper.


\bibliographystyle{plain}

\bibliography{References.bib}
\nocite{andrews_partition_1991}
\end{document}